\documentclass[12pt]{amsart}
\usepackage{amssymb}
\usepackage{commath}
\usepackage{hyperref}
\usepackage{graphics}
\usepackage{bigints}

\makeatletter
\@namedef{subjclassname@2020}{%
  \textup{2020} Mathematics Subject Classification}
\makeatother

\newtheorem{lemma}{Lemma}[section]
\newtheorem{theorem}{Theorem}[section]
\newtheorem{remark}{Remark}[section]

\hypersetup{linkcolor=blue,citecolor=red,filecolor=green,urlcolor=magenta} 

\title[\resizebox{4.6in}{!}{On the sum of the twisted Fourier coefficients of Maass forms by M\"{o}bius function}]{On the sum of the twisted Fourier coefficients of Maass forms by M\"{o}bius function}

\author{K Venkatasubbareddy}
\address{University of Hyderabad\\ 
Hyderabad\\
Telangana\\
India}
\email{20mmpp02@uohyd.ac.in}

\author{A. Kaur}
\address{University of Hyderabad\\ 
Hyderabad\\
Telangana\\
India}
\email{amrinder1kaur@gmail.com}

\author{A. Sankaranarayanan}
\address{University of Hyderabad\\ 
Hyderabad\\
Telangana\\
India}
\email{sank@uohyd.ac.in}

\date{}

\begin{document}

\baselineskip=17pt

\begin{abstract}
In this paper, we study non-trivial upper bounds for the sum $\sum \limits_{n \in S} \abs{\lambda_f(n)}$ where $f$ is a normalized Maass eigencusp form for the full modular group, $\lambda_f(n)$ is the $n$th normalized Fourier coefficient of $f$ and $S$ is a proper subset of positive integers in $[1,x]$ with certain properties.
\end{abstract}



\maketitle

\begin{section}{Introduction} \hfill

\noindent
One of the important questions for several classical non-negative arithmetical functions is to find how often they are small on average. In this connection, assuming Ramanujan conjecture namely
$$ \abs{\alpha_f(p)} = \abs{\beta_f(p)} = 1 \quad \forall \: \text{primes} \: p$$
when $f$ is a normalized Maass eigencusp form,
H. Tang and J. Wu in \cite{HtJw} showed that
\begin{equation} \label{e1}
\sum_{n \leq x} \abs{\lambda_f(n)} \ll_f \frac{x}{(\log x)^{\theta_1}} 
\end{equation}
where $\theta_1 = 0.118 \dots$ and $\lambda_f(n)$ is the $n$th normalized Fourier coefficient of $f$.
They have also obtained results on short intervals. \\

\noindent
Assuming Sato-Tate conjecture for Hecke eigencusp forms (it is proved for holomorphic case by Barnet-Lamb, Geraghty, Harris, \& Taylor in \cite{TbDg}, but is still open for Maass cusp forms), it is also established that (see \cite{HtJw}),
$$ \sum_{n \leq x} \abs{\lambda_f(n^m)} \sim D_m(f) x (\log x)^{- \delta_m} $$
where $D_m(f)$ is a positive constant depending on $m$ and $f$ and
$$ \delta_m := 1- \frac{4(m+1)}{\pi m (m+2)} \cot \left( \frac{\pi}{2(m+1)} \right) .$$ \\

\noindent
There are two important questions:

\begin{enumerate}

\item Is it possible to improve the upper bound of the inequality in (\ref{e1}) even assuming Ramanujan conjecture?

\item Is it possible to give non-trivial improved upper bounds for the sum $\sum \limits_{n \in S} \abs{\lambda_f(n)} $ where $S$ is a proper subset of all the integers in the interval $[1,x]$ such that $|S| \to \infty$ as $x \to \infty$? \\

\end{enumerate}

\noindent
Question 1 seems to be very hard, however regarding Question 2, for some proper subsets $S$, non-trivial upper bounds can be obtained unconditionally. If we obtain $\sum \limits_{n \in S} \abs{\lambda_f(n)} = o(|S|)$, it ensures that on average $\abs{\lambda_f(n)}$ is small with respect to the set $S$. Of course we have compromised on the length of the summand and thus naturally one would expect the upper bound of $\sum \limits_{n \in S} \abs{\lambda_f(n)}$ to go down. Indeed what we show here is that for certain proper subsets of all the integers in $[1,x]$, the upper bound of $\sum \limits_{n \in S} \abs{\lambda_f(n)}$ goes down considerably, that too without Ramanujan conjecture. \\

\noindent
An important result of J. Hoffstein and P. Lockhart in \cite{JhPl} shows that
$$ \sum_{n \leq x} \abs{\lambda_f(n)}^4 \ll x \log x $$
and thus on average $\abs{\lambda_f(n)}^4$ behaves nicely. This will lead to 
$$ \sum_{n \in S} \abs{\lambda_f(n)} \ll_f |S|^{3/4} (x \log x)^{1/4} .$$

\noindent
The question of finding a tight upper bound for the sum $\sum \limits_{n \in S} \abs{\lambda_f(n)}$ becomes more interesting. Thus the main task of this article is to prove:  \\

\begin{theorem} \label{t1}
For a normalized Maass eigencusp form $f$ and $x \geq x_0$ with $x_0$ sufficiently large, the estimate
$$ \sum_{n\leq x} \abs{\lambda_f(n)\mu(n)} \ll_{f,\epsilon} \frac{x(\log \log x)^{5/4}}{\sqrt{\log x}}$$ 
holds unconditionally. \\
\end{theorem}

\begin{remark}
\normalfont
We note that 
$$ \sum_{n \leq x} \abs{\mu(n)} = \frac{x}{\zeta(2)} + O(x^{1/2}) .$$
From Theorem \ref{t1}, we observe that
$$ \frac{\sum \limits_{n \leq x} \abs{\lambda_f(n)\mu(n)}}{\sum \limits_{n \leq x} \abs{\mu(n)}} \ll_{f,\epsilon} \frac{(\log \log x)^{5/4}}{\sqrt{\log x}} $$
which tends to zero as $x \to \infty$ so that the relative density is zero. \\
\end{remark}

\begin{theorem} \label{t2}
Let $S_k$ be the set of all $k$-free integers ($k \geq 3$) in the interval $[1,x]$, then the inequality
$$ \sum_{n \in S_k} \abs{\lambda_f(n)} \ll_{f,k,\epsilon} \frac{x(\log \log x)^{5/4}}{\sqrt{\log x}} = o \left( |S_k| \right)$$
holds unconditionally. \\
\end{theorem}

\begin{remark} \label{r2}
\normalfont
It is not difficult to see from the arguments of the paper that the proof goes through very well for any proper subset $S \subset \{ 1,2, \dots, [x] \}$ with the cardinality $|S|$ satisfying
$$ \frac{x}{(\log x)^{1/2 - \eta}} \leq |S| \leq (1- \eta) x $$
($\eta$ is any small positive constant) so that $\abs{\lambda_f(n)}$ is small on average with respect to $S$. However, there are some important proper subsets for which these arguments do not provide the desired results. We exhibit two such examples in the last section of this paper. \\
\end{remark}

\subsection*{Main Idea :} \hfill\\

\noindent
First we study the cognated sum $\sum \limits_{n \in S} \frac{\abs{\lambda_f(n)}}{n}$ by splitting it into two sums pertaining to $\mathcal{L}$-smooth and its compliment to get non-trivial upper bounds. Then by the Lemma \ref{l6}, we pass onto $\sum \limits_{n \in S} \abs{\lambda_f(n)}$. The whole point here is that we can avoid the Ramanujan conjecture in these situations we consider. We treat the squarefree set case in detail and give the sketch of the proof in the general $k$-free set case. \\

\subsection*{Relation between sums $\sum \limits_{n \leq x} g(n)$ and $\sum \limits_{n \leq x} \frac{g(n)}{n}$ :} \hfill\\

\noindent
Let $g(n)$ be a real non-negative arithmetic function. We are interested in the size of the sums
$$ S(x) = \sum_{n \leq x} g(n) \quad \text{and} \quad L(x) = \sum_{n \leq x} \frac{g(n)}{n}. $$
Trivially, $S(x) \leq x L(x)$ and Riemann Stieltjes integration gives the relation
$$ L(x) = \int_1^x \frac{d S(u)}{u} \leq \frac{S(x)}{x} - \frac{S(1^+)}{1^+} + \int_{1^+}^x \frac{\abs{S(u)}}{u^2} du .$$ \\

\noindent
In 1980, Shiu \cite{Ps} obtained a general upper bound for short sums of functions satisfying certain properties: \\ 
Let $\alpha, \beta \in [0,1]$ and let $x, y$ satisfy $x \geq y \geq x^{\alpha}$. Then for positive integers $a, q$ with $(a,q)=1$ we have
$$ \sum_{x < n \leq x+y \atop n \equiv a (mod q)} g_1(n) \ll \frac{y}{\phi(q) \log x} \exp \bigg\{ \sum_{p \leq x \atop p \nmid q} \frac{g_1(p)}{p} \bigg\}  $$
uniformly for $1 \leq q \leq x^{\beta}$. \\

\noindent
Later in 1998, Nair and Tenenbaum in \cite{MnGt} gave an interesting inequality connecting the two sums $S(x)$ and $L(x)$ for a class of non-negative arithmetic functions satisfying some conditions. \\

\noindent
Let $F(n)$ be a non-negative arithmetic function such that
$$ F(mn) \leq \min \left( D^{\omega(m)}, Em^{\epsilon} \right) F(n) $$
for all $m,n$ with $(m,n)=1$ and any $D \geq 1$, $E \geq 1$. Here $\omega(m)$ denotes the total number of prime factors of $m$, counted with multiplicity. Suppose $Q \in \mathbb{Z}[X]$ is an irreducible polynomial and $\rho(m) = \rho_Q(m)$ denotes the number of roots of $Q$ in $\mathbb{Z}/m\mathbb{Z}$.

\noindent
Then, a special case of their result gives
\begin{equation} \label{e2}
\sum_{x < n \leq x+y} F \left( \abs{Q(n)} \right) \ll y \prod_{p \leq x} \left( 1-\frac{\rho(p)}{p} \right) \sum_{n \leq x} \frac{F(n) \rho(n)}{n}
\end{equation} 
uniformly for $x^{\alpha} \leq y \leq x$ with $x$ sufficiently large and where $\epsilon$ and $\alpha$ can be arbitrary small positive real numbers satisfying certain conditions. \\

\noindent
For some simplified result of the form (\ref{e2}), we refer to Lemma 9.6 of De Koninck and Luca in \cite{JdFl}. \\

\end{section}

\begin{section}{Preliminaries and Notations} \hfill

\noindent
Let $n \geq 2$, and let $v=(v_1,v_2,\dots,v_{n-1}) \in \mathbb{C}^{n-1}$. A Maass form (see \cite{Dg}) for $SL(n,\mathbb{Z})$ of type $v$ is a smooth function $f \in \mathcal{L}^2(SL(n,\mathbb{Z}) \backslash \mathcal{H}^n)$ which satisfies
\begin{enumerate}
\item $f(\gamma z) = f(z)$, for all $\gamma \in SL(n,\mathbb{Z}), z \in \mathcal{H}^n$,
\item $Df(z) = \lambda_D f(z)$, for all $D \in \mathcal{D}^n$,
\item $\int \limits_{(SL(n,\mathbb{Z}) \cap U) \backslash U} f(uz) du = 0, $\\
for all upper triangular groups $U$ of the form 
$$U = \left\{
\begin{pmatrix} 
I_{r_1} &             &           &  \\
           & I_{r_2}  &           & *\\ 
           &             & \ddots &   \\ 
           &             &           & I_{r_b}  \\ 
\end{pmatrix} \right\}, $$
with $r_1+r_2+\cdots+r_b=n$. Here, $I_r$ denotes the $r \times r$ identity matrix, and $*$ denotes arbitrary real entries. \\

\end{enumerate}

\noindent
Let $M^*(\Gamma)$ be the set of normalized Maass eigencusp forms for the full modular group $\Gamma=SL(2,\mathbb{Z})$ and $f \in M^*(\Gamma)$. \\

\noindent
Denote by $\lambda_f(n)$ the $n$th normalized Fourier coefficient of $f$ and also the eigenvalue of $f$ under the Hecke operator $T_n$. \\

\noindent
From the Hecke theory, $\lambda_f(n)$ satisfies the multiplicative relation
$$ \lambda_f(m) \lambda_f(n) = \sum_{d|(m,n)} \lambda_f \left( \frac{mn}{d^2} \right) $$
for all integers $m \geq 1$ and $n \geq 1$. \\

\noindent
Thus for each prime number $p$ there are two complex numbers $\alpha_f(p)$ and $\beta_f(p)$ such that
$$ \alpha_f(p) \beta_f(p) =1 $$
and
$$ \lambda_f(p^{\nu}) = \alpha_f(p)^{\nu} + \alpha_f(p)^{\nu-1} \beta_f(p) + \cdots + \beta_f(p)^{\nu} $$
for all integers $\nu \geq 1$. \\

\noindent
Ramanujan conjecture states that
$$ \abs{\alpha_f(p)} = \abs{\beta_f(p)} =1 $$
for all primes $p$. \\

\noindent
Unconditionally, we only have
$$ p^{-7/64} \leq \abs{\alpha_f(p)} \leq p^{7/64} $$
$$ p^{-7/64} \leq \abs{\beta_f(p)} \leq p^{7/64} $$
for all primes $p$, due to Kim and Sarnak \cite{Hk}. \\

\noindent
We use the following standard notations:

\begin{enumerate}

\item We write $\log_1 x := \log x$ and for $k \geq 2$, $\log_k x := \log \log_{k-1} x $.

\item We denote the largest prime factor of an integer $n$ by $P^+(n)$ and the smallest prime factor of an integer $n$ by $P^-(n)$ (with the convention that $P^+(1)=P^-(1)=0$). 

\item Let $\mathcal{L} := \mathcal{L}(x)=(\log x)^2$ where $x$ is sufficiently large.

\item Constants $C$ with indices are some fixed positive constants. $\epsilon$ is any small positive constant.  

\item Implied constants might depend atmost on $f$ and $\epsilon$ in the squarefree set case and on $f,\ \epsilon$ and $k$ in the $k$-free set case.

\end{enumerate}

\end{section}

\begin{section}{Some Lemmas} \hfill

\begin{lemma} \label{l1}
Let $x$ be sufficiently large and $\frac{1}{2}<a<1$. Then there exists a fixed constant $C_1 > 0$ such that
$$\sum_{n \leq x \atop P^+(n)\leq \mathcal{L}} n^{-a} \ll x^{1/2-a}\exp \left( C_1 \frac{\log x}{\log_2 x} \right). $$ \\
\end{lemma}

\begin{proof}

Let $\psi(x,y) = \sum \limits_{n \leq x \atop P^+(n) \leq y} 1$ and $y \geq \mathcal{L} = (\log x)^2$. Then for $u=\frac{\log x}{\log y}$,

$$ \psi(x,y) = x \rho (u) \exp \left( O \left( u \exp \left( -(\log u)^{3/5-o(1)} \right) \right) \right) \quad \text{if} \quad y \geq (\log x)^{1+\epsilon} $$
where $\rho$ is Dickmann's function (see \cite{Nd,Kd}) defined by
$$ \rho(u) =1 \quad \text{for} \quad 0 \leq u \leq 1 $$
$$ \rho'(u) = -\frac{1}{u} \rho(u-1) \quad \text{for} \quad u>1 .$$

\noindent
Norton (see \cite{Kn}) has shown that (as $u \to \infty$)

$$ \rho(u) = \exp \left( -u \left( \log u + \log_2 u -1 + \frac{\log_2 u}{\log u} + O \left( \frac{1}{\log u} \right) \right) \right) .$$

\noindent
We obtain
$$ \psi(x,y) \leq x \exp \left( -u \left( \log u + \log_2 u -1 \right) \right) \quad \text{for} \quad y \geq (\log x)^{1+\epsilon} . $$ \\

\noindent
Using integration by parts (following arguments as in \cite{HmAs}), we obtain

\begin{align*}
S_1 &:= \sum_{n\leq x \atop P^+(n)\leq \mathcal{L}} n^{-a}\\
&= \psi(x,\mathcal{L}) x^{-a} + a \int_1^x \psi(w,\mathcal{L}) w^{-1-a} \ dw\\
&\ll 1+\int_1^x w^{-a} \exp \bigg( -\bigg(\frac{\log w}{\log \mathcal{L}}\bigg)\bigg\{\log\bigg(\frac{\log w}{\log \mathcal{L}}\bigg) \\
&\qquad+\log_2\bigg(\frac{\log w}{\log\mathcal{L}}\bigg)-1 \bigg\}\bigg) \ dw.
\end{align*}
Partition the interval of integration into subintervals of the form $[x e^{-(k+1)},\ xe^{-k})$ for $0\leq k\leq \log x$, then
\begin{align*}
J_k &= \int_{xe^{-(k+1)}}^{xe^{-k}} \psi(w,\mathcal{L}) w^{-1-a}\ dw \\
&\ll \big( xe^{-k} \big)^{1-a} \exp \bigg( -\bigg( \frac{\log x-(k+1)}{2\log_2 x} \bigg) \bigg\{ \log \bigg( \frac{\log x-(k+1)}{2\log_2 x} \bigg) \\
&\qquad+\log_2 \bigg( \frac{\log x-(k+1)}{2\log_2 x} \bigg) -1 \bigg\} \bigg) \\
&\ll x^{1-a} e^{-k(1-a)} \exp \bigg(-\bigg( \frac{\log x-(k+1)}{2\log_2 x} \bigg) \bigg\{ \log \bigg( \frac{\log x-(k+1)}{2\log_2 x} \bigg) \\
&\qquad + \log_2 \bigg( \frac{\log x-(k+1)}{2\log_2 x} \bigg)-1 \bigg\} \bigg) \\
&\ll x^{1-a} e^{-k(1-a)} x^{-1/2} \exp \bigg( \frac{\log 2}{2}\frac{\log x}{\log_2x} \bigg) \\
&=x^{1/2-a}e^{-k(1-a)} \exp \bigg( \frac{\log2}{2} \frac{\log x}{\log_2x} \bigg).
\end{align*}

\noindent
Now summing over $k$, we get 
$$ S_1 = \sum_{n\leq x \atop P^+(n)\leq \mathcal{L}} n^{-a} \ll x^{1/2-a} \exp \bigg( C_1 \frac{\log x}{\log_2 x} \bigg). $$ 

\end{proof}
\hfill

\begin{lemma} \label{l2}
For $x$ sufficiently large, we have
$$ \sum_{n\leq x \atop P^+(n) \leq \mathcal{L}} \frac{\abs{\lambda_f(n) \mu(n)}}{n} \ll_\epsilon x^{-\frac{25}{64}+\epsilon} \exp \bigg( C_1 \frac{\log x}{\log_2 x} \bigg). $$ \\
\end{lemma}

\begin{proof}

\noindent
Using the unconditional bound for $\abs{\lambda_f(n)}$, we get
\begin{align*}
\sum_{n\leq x \atop P^+(n) \leq \mathcal{L}} \frac{\abs{\lambda_f(n) \mu(n)}}{n}
&\leq \sum_{n\leq x \atop P^+(n) \leq \mathcal{L}} \frac{\abs{\lambda_f(n)}}{n}\\
&\leq \sum_{n\leq x \atop P^+(n) \leq \mathcal{L}} \frac{n^{\frac{7}{64}} d(n)}{n}\\
&\ll_\epsilon \sum_{n\leq x \atop P^+(n) \leq \mathcal{L}} n^{-1+\frac{7}{64}+\epsilon}\\
&= \sum_{n\leq x \atop P^+(n) \leq \mathcal{L}} n^{-(1-\frac{7}{64}-\epsilon)}.
\end{align*}
Taking $a=1-\frac{7}{64}-\epsilon$ in Lemma \ref{l1} (note that $ \frac{1}{2} < a <1$), we get 
\begin{align*}
\sum_{n\leq x \atop P^+(n)\leq\mathcal{L}} \frac{\abs{\lambda_f(n) \mu(n)}}{n} & \ll_\epsilon x^{\frac{1}{2}-1+\frac{7}{64}+\epsilon} \exp \bigg( C_1 \frac{\log x}{\log_2x} \bigg) \\
&\ll_\epsilon x^{-\frac{25}{64}+\epsilon} \exp \bigg( C_1 \frac{\log x}{\log_2x} \bigg). 
\end{align*}
\end{proof}
\hfill

\begin{lemma} \label{l3}
For $x$ sufficiently large, we have
$$ \sum_{n \leq x \atop P^-(n) > \mathcal{L}} \frac{1}{n} \ll \log x. $$ \\
\end{lemma}

\begin{proof}
Trivially,
\begin{align*}
\sum_{n \leq x \atop P^-(n) > \mathcal{L}} \frac{1}{n} 
&\leq \sum_{n \leq x} \frac{1}{n} \\
&\leq 1 + \int_1^x \frac{du}{u} \\
&\leq 1 + \log x \\
&\ll \log x.
\end{align*}
\end{proof}
\hfill

\begin{lemma} \label{l4}
We have
$$ \sum_{n\leq x \atop P^-(n)>\mathcal{L}} \frac{\abs{\lambda_f(n)}^4}{n} \ll \log^2 x. $$ \\
\end{lemma}

\begin{proof}
For $\mathcal{L} < L_2 \leq x$, we get
\begin{align*}
\sum_{n\leq x \atop P^-(n)>\mathcal{L}} \frac{\abs{\lambda_f(n)}^4}{n} &= \int_{L_2}^x \frac{1}{u}\ d \bigg( \sum_{n\leq x \atop P^-(n)>\mathcal{L}(u)} \abs{\lambda_f(n)}^4 \bigg)  \\
&= \left. \frac{\sum \limits_{n\leq u \atop P^-(n)>\mathcal{L}(u)} \abs{\lambda_f(n)}^4 }{u} \right|_{L_2}^x + \bigints_{L_2}^x \frac{\bigg( \sum \limits_{n\leq u \atop P^-(n)>\mathcal{L}(u)} \abs{\lambda_f(n)}^4 \bigg) }{u^2}\ du .\\
\end{align*}

\noindent
From \cite{JhPl}, we observe that there exists non-negative coefficients $\lambda^*(n)$ such that
\begin{equation} \label{e3}
\sum_{n \leq x} \abs{\lambda_f(n)}^4 \leq \sum_{n \leq x} \lambda^*(n) \ll_f x \log x 
\end{equation}
where
$$ \sum_{n=1}^{\infty} \frac{\lambda^*(n)}{n^s} := L(s,{\rm{sym}}^4 f) L^3(s,{\rm{sym}}^2 f) \zeta^2(s) $$
in $\Re(s)>1$.
Hence, 

\begin{align*}
\sum_{n\leq x \atop P^-(n)>\mathcal{L}} \frac{\abs{\lambda_f(n)}^4}{n} &= O \big( \log x \big) + O \big( \log_2 x \big) + O \bigg( \int_{L_2}^x \frac{u\log u}{u^2}\ du \bigg) \\
&\ll \bigg( \int_{L_2}^x \frac{\log u}{u}\ du \bigg) + \log x \\
&\ll \big( \log^2 x \big).
\end{align*}
\end{proof}
\hfill

\begin{lemma} \label{l5}
For sufficiently large $x$, we get
$$ \sum_{n\leq x \atop P^-(n) > \mathcal{L}} \frac{\abs{\lambda_f(n)\mu(n)}}{n} \ll (\log x)^{\frac{5}{4}}. $$ \\
\end{lemma}

\begin{proof}
We have
\begin{align*}
\sum_{n\leq x \atop P^-(n)>\mathcal{L}} \frac{\abs{\lambda_f(n)\mu(n)}}{n} &\leq \sum_{n\leq x \atop P^-(n) > \mathcal{L}} \frac{\abs{\lambda_f(n)}}{n} \\
&= \sum_{\substack{n\leq x \\ P^-(n)>\mathcal{L} \\ \abs{\lambda_f(n)} \leq M}} \frac{\abs{\lambda_f(n)}}{n} + \sum_{\substack{n\leq x \\ P^-(n) > \mathcal{L} \\ \abs{\lambda_f(n)} > M}} \frac{\abs{\lambda_f(n)}}{n} \\
&= S_1+S_2
\end{align*}
where $M$ is an open positive quantity which might depend on $x$. \\
We get
\begin{align*}
S_1&=\sum_{\substack{n\leq x\\P^-(n)>\mathcal{L}\\\mid\lambda_f(n)\mid\leq M}}\frac{\mid\lambda_f(n)\mid}{n}\\
&\leq \sum_{\substack{n\leq x\\P^-(n)>\mathcal{L}}}M\frac{1}{n}\\
&\ll M \log x
\end{align*}
using Lemma \ref{l3}.

\noindent
Now,
\begin{align*}
S_2 &= \sum_{\substack{n\leq x \\ P^-(n) > \mathcal{L} \\  \abs{\lambda_f(n)} > M}} \frac{\abs{\lambda_f(n)}}{n} \\
&=\sum_{\substack{n\leq x \\ P^-(n) > \mathcal{L} \\ \abs{\lambda_f(n)} > M}} \frac{\abs{\lambda_f(n)}^4}{\abs{\lambda_f(n)}^3}\frac{1}{n} \\
&\leq \sum_{\substack{n\leq x \\ P^-(n) > \mathcal{L} \\ \abs{\lambda_f(n)} > M}} \frac{1}{M^3} \frac{\abs{\lambda_f(n)}^4}{n}.
\end{align*}

\noindent
So, using Lemma \ref{l4},
$$ S_2 \ll \frac{\log^2 x}{M^3}. $$
Now choose $M$ such that 
$$ M \log x \sim \frac{\log^2x}{M^3} $$
i.e., 
$$M \sim \{ (\log x) \}^\frac{1}{4}.$$

\noindent
Hence,
\begin{align*}
\sum_{n\leq x \atop P^-(n)>\mathcal{L}} \frac{\abs{\lambda_f(n)}}{n} 
&\ll (\log x)^\frac{5}{4}. 
\end{align*}
\end{proof}
\hfill

\begin{lemma} \label{l6}
Let $g$ be a multiplicative function such that $g(n) \geq 0$ for all $n$, and such that there exists constants $A$ and $B$ such that for all $x>1$ both inequalitites
$$ \sum_{p \leq x} g(p) \log p \leq A x h_1(x) $$
and
$$ \sum_p \sum_{\alpha \geq 2} \frac{g(p^{\alpha})}{p^{\alpha}} \log p^{\alpha} \leq B h_2(x) $$
hold where $h_1(x)$(increasing) and $h_2(x)$ are positive functions of $x$ for all $x \geq 1$. Then for $x>1$, we have
$$ \sum_{n \leq x} g(n) \leq (A h_1(x) + B h_2(x) +1) \frac{x}{\log x} \sum_{n \leq x} \frac{g(n)}{n}. $$ \\
\end{lemma}

\begin{proof}
We follow the arguments as in Lemma 9.6 of \cite{JdFl}. \\

\noindent
Let $S(x) = \sum \limits_{n \leq x} g(n)$ and $L(x) = \sum \limits_{n \leq x} \frac{g(n)}{n}.$ 
Then
$$ L(x) = \sum_{n \leq x} \frac{g(n)}{n} \geq \frac{1}{x} \sum_{n \leq x} g(n) = \frac{1}{x} S(x) $$
i.e., $S(x) \leq x L(x)$.
\begin{align*}
S(x) \log x &= \sum_{n \leq x} g(n) \log x \\
&= \sum_{n \leq x} g(n) \log \left( \frac{x}{n} \right) + \sum_{n \leq x} g(n) \sum_{p||n} \log p + \sum_{n \leq x} g(n) \sum_{\alpha \geq 2 \atop p^{\alpha}||n } \log p^{\alpha} \\
&= S_1 + S_2 + S_3.
\end{align*}
\begin{align*}
S_1 &= \sum_{n \leq x} g(n) \log \left( \frac{x}{n} \right) \\
&\leq \sum_{n \leq x} g(n) \frac{x}{n} \\
&= x \sum_{n \leq x} \frac{g(n)}{n} \\
& \leq x L(x) .
\end{align*}

\noindent
Write $n = mp$ such that $p \nmid m$ in $S_2$. 

\begin{align*}
S_2 &= \sum_{n \leq x} g(n) \sum_{p||n} \log p \\
&= \sum_{mp \leq x} g(mp) \sum_{p \nmid m} \log p \\
&= \sum_{m \leq x} g(m) \sum_{p \leq x/m \atop p \nmid m} g(p) \log p \\
&\leq \sum_{m \leq x} g(m) \sum_{p \leq x/m} g(p) \log p \\
&\leq \sum_{m \leq x} g(m) A \frac{x}{m} h_1 \left( \frac{x}{m} \right) \\
&\leq A x h_1(x) \sum_{m \leq x} \frac{g(m)}{m} \\
&= Ax h_1(x) L(x) .
\end{align*}

\noindent
Write $n = mp^{\alpha}$ such that $p \nmid m$ in $S_3$.

\begin{align*}
S_3 &= \sum_{n \leq x} g(n) \sum_{\alpha \geq 2 \atop p^{\alpha} || n} \log p^{\alpha} \\
&= \sum_{mp^{\alpha} \leq x} g(mp^{\alpha}) \sum_{\alpha \geq 2 \atop p \nmid m} \log p^{\alpha} \\
&= \sum_p \sum_{\alpha \geq 2} g(p^{\alpha}) \log p^{\alpha} \sum_{m \leq x/p^{\alpha} \atop p \nmid m} g(m) \\
&\leq \sum_p \sum_{\alpha \geq 2} g(p^{\alpha}) \log p^{\alpha} S \left( \frac{x}{p^{\alpha}} \right) \\
&\leq \sum_p \sum_{\alpha \geq 2} g(p^{\alpha}) \log p^{\alpha} \frac{x}{p^{\alpha}} L \left( \frac{x}{p^{\alpha}} \right) \\
&\leq x L(x) \sum_p \sum_{\alpha \geq 2} \frac{g(p^{\alpha})}{p^{\alpha}} \log p^{\alpha} \\
&\leq x L(x) B h_2(x).
\end{align*}
\noindent
Finally,
\begin{align*}
S(x) \log x &= S_1 + S_2 + S_3 \\
&\leq \left( 1+ A h_1(x) +B h_2(x) \right) x L(x). \\
\end{align*}

\noindent
Therefore,
$$ S(x) \leq \left( A h_1(x) + B h_2(x) + 1 \right) \frac{x}{\log x} L(x) .$$

\end{proof}
\hfill

\begin{lemma} \label{l7}
For all $x>1$, there exists a positive constant $A$ such that
$$ \sum_{p\leq x} \abs{\lambda_f(p) \mu(p)} \log p \leq Ax \sqrt{\log x} $$ \\
and
$$ \sum_p \sum_{\alpha \geq 2} \frac{\abs{\lambda_f(p^{\alpha}) \mu(p^{\alpha})} \log p^{\alpha}}{p^{\alpha}} = 0.$$
\end{lemma}

\begin{proof}
We get
\begin{align*}
\sum_{p\leq x} \abs{\lambda_f(p)\mu(p)} \log p &\leq \sum_{p\leq x} \abs{\lambda_f(p)} \log p \\
&= \sum_{p\leq x \atop \abs{\lambda_f(p)} \leq L} \abs{\lambda_f(p)} \log p + \sum_{p\leq x \atop \abs{\lambda_f(p)}>L} \frac{\abs{\lambda_f(p)}^4}{\abs{\lambda_f(p)}^3} \log p \\
&\leq L \sum_{p\leq x} \log p + \frac{\log x}{L^3} \sum_{p\leq x} \abs{\lambda_f(p)}^4 \\
&\leq L \sum_{p\leq x} \log p + \frac{\log x}{L^3} \sum_{n\leq x} \abs{\lambda_f(n)}^4 \\
&\leq L \sum_{p\leq x} \log p + \frac{\log x}{L^3} \sum_{n\leq x} \lambda^*(n) \\
&\ll_f  x L + x \frac{(\log x)^2}{L^3}
\end{align*}
by prime number theorem and inequality (\ref{e3}). \\

\noindent
Now choose $L$ such that
$$ x L \sim x \frac{(\log x)^2}{L^3} $$
i.e.,  
$$ L \sim \sqrt{\log x}. $$
Hence,
$$ \sum_{p\leq x} \abs{\lambda_f(p) \mu(p)} \log p \leq Ax \sqrt{\log x}. $$
Second result follows trivially since $\mu(p^{\alpha})=0$ for $\alpha \geq 2$.
\end{proof}
\hfill

\begin{lemma} \label{l8}
We have
$$ \sum_{n\leq x} \abs{\lambda_f(n)\mu(n)} \leq C_2 \frac{x}{\sqrt{\log x}} \sum_{n\leq x} \frac{\abs{\lambda_f(n)\mu(n)}}{n}. $$ \\
\end{lemma}

\begin{proof}
The result follows by taking $g(n) = \abs{\lambda_f(n) \mu(n)}$ in Lemma \ref{l6} using Lemma \ref{l7}.
\end{proof}
\hfill

\end{section}

\begin{section}{Proof of Theorem \ref{t1}} \hfill

\noindent
First we note that (with $\mathcal{L} = (\log x)^2$),
\begin{align*}
\sum_{n\leq \mathcal{L}} \frac{\abs{\lambda_f(n)\mu(n)}}{n} &\leq \sum_{n \leq \mathcal{L}} \frac{\abs{\lambda_f(n)}}{n} \\
&\leq \left( \sum_{n \leq \mathcal{L}} \frac{\abs{\lambda_f(n)}^4}{n} \right)^{1/4} \left( \sum_{n \leq \mathcal{L}} \frac{1}{n} \right)^{3/4} \\
&\leq C_3 (\log \mathcal{L})^{1/2} (\log \mathcal{L})^{3/4} \\
&\leq C_3 (\log \mathcal{L})^{5/4} \\
&\leq C_4 (\log \log x)^{5/4}. \\
\end{align*}

\noindent
For $n>\mathcal{L}$, we write $n=m_1 m_2$ where $m_1 = \prod \limits_{p|n \atop p \leq \mathcal{L}} p$. Thus,

\begin{align*}
\sum_{\mathcal{L} < n \leq x} \frac{\abs{\lambda_f(n)\mu(n)}}{n} &= \sum_{m_1 \leq x \atop P^+(m_1) \leq \mathcal{L}} \frac{\abs{\lambda_f(m_1)\mu(m_1)}}{m_1} \sum_{m_2 \leq \frac{x}{m_1} \atop P^-(m_2) > \mathcal{L}} \frac{\abs{\lambda_f(m_2)\mu(m_2)}}{m_2} \\
&\ll \sum_{m_1 \leq x \atop P^+(m_1) \leq \mathcal{L}}\frac{\abs{\lambda_f(m_1)\mu(m_1)}}{m_1} \left( \log \left( \frac{x}{m_1} \right) \right)^\frac{5}{4} \\
&\ll x^{-\frac{25}{64}+\epsilon} \exp \bigg( C_1 \frac{\log x}{\log_2x} \bigg) (\log x)^\frac{5}{4}
\end{align*}
using Lemmas \ref{l2} and \ref{l5}.

\noindent
By Lemma \ref{l8}, we have
\begin{align*}
\sum_{n\leq x} \abs{\lambda_f(n)\mu(n)} & \leq C_2 \frac{x}{\sqrt{\log x}} \sum_{n\leq x} \frac{\abs{\lambda_f(n)\mu(n)}}{n} \\
&\leq C_5 \frac{x}{\sqrt{\log x}} \left\{ (\log \log x)^{5/4} + x^{-{25/64}+\epsilon} \exp \bigg( C_1 \frac{\log x}{\log_2x} \bigg) (\log x)^{5/4} \right\} \\
&\ll_f \frac{x (\log \log x)^{5/4}}{\sqrt{\log x}} .
\end{align*}
This completes the proof of Theorem \ref{t1}.

\end{section}
\hfill\\

\begin{section}{Sketch of proof of Theorem \ref{t2}} \hfill\\

\noindent
A natural number $n = p_1^{a_1} p_2^{a_2} \dots p_r^{a_r}$ is called $k$-free if $a_i \leq k-1 \: \forall \: i=1,2,\dots,r.$ Let

\begin{equation*}
h_k(n) = 
\begin{cases}
1 \quad \text{$n$ is $k$-free}, \\
0 \quad \text{otherwise}. \\
\end{cases}
\end{equation*}

\noindent
Therefore,
\begin{align*}
\sum_{n=1}^{\infty} \frac{h_k(n)}{n^s} &= \prod_p (1+p^{-s} + \cdots + p^{-(k-1)s}) \\
&= \prod_p \frac{1-p^{-ks}}{1-p^{-s}} \\
&= \frac{\zeta(s)}{\zeta(ks)}. \\
\end{align*}

\noindent
We know that(see \cite{Av}),
$$ \sum_{n \leq x} h_k(n) = \frac{x}{\zeta(k)} + O \left( x^{1/k} \exp \left( -C_6 (\log x)^{3/5}(\log \log x)^{-1/5} \right) \right) .$$

\noindent
Similar to the proof of Theorem \ref{t1}, we have 
$$ \sum_{n \leq \mathcal{L}} \frac{\abs{\lambda_f(n)}h_k(n)}{n} \ll_{f,k} (\log \log x)^{5/4}. $$ \\

\noindent
For $n> \mathcal{L}$, we write $n=m_1 m_2$ where $m_1 = p_1^{a_1} p_2^{a_2} \dots p_r^{a_r}$ and $m_2 = q_1^{b_1} q_2^{b_2} \dots q_s^{b_s}$ with $a_i \leq k-1$ for all $i=1,2,\dots,r$ and $b_j \leq k-1$ for all $j=1,2,\dots,s$ such that $p_i \leq \mathcal{L}$ and $q_j > \mathcal{L}$. Hence, $(m_1,m_2)=1$. Thus,

\begin{align*}
 \sum_{\mathcal{L} < n \leq x} \frac{\abs{\lambda_f(n)}h_k(n)}{n} 
&\leq \sum_{m_1 \leq x \atop P^+(m_1) \leq \mathcal{L}} \frac{\abs{\lambda_f(m_1)}h_k(m_1)}{m_1} \sum_{m_2 \leq \frac{x}{m_1} \atop P^-(m_2) > \mathcal{L}} \frac{\abs{\lambda_f(m_2)}h_k(m_2)}{m_2} \\
&\ll_{f,k} \sum_{m_1 \leq x \atop P^+(m_1) \leq \mathcal{L}}\frac{\abs{\lambda_f(m_1)}h_k(m_1)}{m_1} \left( \log \left( \frac{x}{m_1} \right) \right)^\frac{5}{4} \\
&\ll_{f,k,\epsilon} x^{-\frac{25}{64}+\epsilon} \exp \bigg( C_1 \frac{\log x}{\log_2x} \bigg) (\log x)^\frac{5}{4}
\end{align*}
using arguments similar to that of Lemmas \ref{l2} and \ref{l5}. \\

\noindent
We have,
\begin{align*}
\sum_p \sum_{\alpha \geq 2} \frac{\abs{\lambda_f(p^{\alpha})}h_k(p^{\alpha}) \log p^{\alpha}}{p^{\alpha}} & \leq \sum_p \sum_{\alpha \geq 2} \frac{\abs{\lambda_f(p^{\alpha})} \log p^{\alpha}}{p^{\alpha}} \\
&\leq \sum_p \sum_{\alpha \geq 2} \frac{p^{\frac{7}{64} \alpha} (\alpha+1) \alpha \log p}{p^{\alpha}} \\
&\leq \sum_p \sum_{\alpha \geq 2} \frac{2 \alpha^2 \log p}{p^{\frac{57}{64} \alpha}} \\
&= \sum_p \sum_{\alpha \geq 2} \frac{2 \alpha^2 \log p}{ p^{\frac{7}{8} \alpha} e^{\frac{\alpha}{64} \log p}} \\
& \leq \sum_p \sum_{\alpha \geq 2} \frac{2 \alpha^2 \log p}{p^{\frac{7}{8} \alpha} \frac{(\frac{\alpha}{64} \log p)^2}{2!}} \\
& \ll \sum_p \sum_{\alpha \geq 2} \frac{1}{p^{\frac{7}{8} \alpha} \log p} \\
& \ll \sum_p \sum_{\alpha \geq 2} \frac{1}{p^{\frac{7}{8} \alpha}} \\
& = \sum_p \frac{\frac{1}{p^{14/8}}}{1-\frac{1}{p^{7/8}}} \\
& = \sum_p \frac{1}{p^{7/8}(p^{7/8} -1)} \\
& \ll B. 
\end{align*}

\noindent
Using similar result as in Lemma \ref{l8}, we get
\begin{align*}
\sum_{n \in S_k} \abs{\lambda_f(n)} &= \sum_{n \leq x} \abs{\lambda_f(n)}h_k(n) \\
&\ll_{f,k} \frac{x}{\sqrt{\log x}} \sum_{n \leq x} \frac{\abs{\lambda_f(n)}h_k(n)}{n} \\
&\ll_{f,k,\epsilon} \frac{x}{\sqrt{\log x}} \left\{ (\log \log x)^{5/4} + x^{-{25/64}+\epsilon} \exp \bigg( C_1 \frac{\log x}{\log_2x} \bigg) (\log x)^{5/4} \right\} \\
&\ll_{f,k,\epsilon} \frac{x (\log \log x)^{5/4}}{\sqrt{\log x}}  
\end{align*}
which completes the proof.
\end{section}
\hfill\\

\begin{section}{Concluding remarks} \hfill\\

\noindent
In this section, we discuss two examples where a similar analysis gives an upper bound for $\sum \limits_{n \in S} \abs{\lambda_f(n)}$ which is not of the order $o(|S|)$. Hence, we are not able to ensure that $\abs{\lambda_f(n)}$ assumes smaller values on average on these sets. First is the primes set  and second is the squarefull numbers set. \\

\noindent
Define
\begin{equation*}
\chi_1(n) = 
\begin{cases}
1 \quad \text{if n is a prime}, \\
0 \quad \text{otherwise}.
\end{cases}
\end{equation*}

\noindent
Note that
$$ \sum_{n \leq x} \chi_1(n) = \pi(x) \ll \frac{x}{\log x} .$$ \\

\noindent
Then, 
\begin{align*}
\sum_{n \in P} \abs{\lambda_f(n)} &= \sum_{1 \leq n \leq x} \abs{\lambda_f(n)} \chi_1(n) \\
& \leq \left( \sum_{n \leq x} \abs{\lambda_f(n)}^4 \right)^{1/4} \left( \sum_{n \leq x} \left( \chi_1(n) \right)^{4/3} \right)^{3/4} \\
& \ll_f \left( \sum_{n \leq x} \lambda^*(n) \right)^{1/4} \left( \sum_{p \leq x} 1 \right)^{3/4} \\
& \ll_f ( x \log x )^{1/4} \left( \frac{x}{\log x} \right)^{3/4} \\
& \ll_f \frac{x}{\sqrt{\log x}}. \\
\end{align*}

\noindent
A number $n = p_1^{a_1} p_2 ^{a_2} \dots p_r^{a_r}$ is called a squarefull number if $a_i \geq 2$ for all $i=1,2,\dots,r$. Let
\begin{equation*}
\chi_2(n) = 
\begin{cases}
1 \quad \text{if n is squarefull}, \\
0 \quad \text{otherwise}.
\end{cases}
\end{equation*}

\noindent
From \cite{Av}, we have
$$ \sum_{n \leq x} \chi_2(n) = \frac{\zeta(3/2)}{\zeta(3)} x^{1/2} + O(x^{1/3}) .$$ \\

\noindent
Let $S^*$ denotes the set of squarefull numbers. Then, 
\begin{align*}
\sum_{n \in S^*} \abs{\lambda_f(n)} &= \sum_{1 \leq n \leq x} \abs{\lambda_f(n)} \chi_2(n) \\
& \leq \left( \sum_{n \leq x} \abs{\lambda_f(n)}^4 \right)^{1/4} \left( \sum_{n \leq x} \left( \chi_2(n) \right)^{4/3} \right)^{3/4} \\
& \ll_f \left( \sum_{n \leq x} \lambda^*(n) \right)^{1/4} \left( x^{1/2} \right)^{3/4} \\
& \ll_f ( x \log x )^{1/4} x^{3/8} \\
& \ll_f x^{5/8} (\log x)^{1/4}. \\
\end{align*}

\noindent
It is important to note that in these two cases, the study of cognated sums $S(x)$ and $L(x)$ will actually lead to weaker estimates than what H\"{o}lder's inequality would give. Thus we observe that the averaging result in (\ref{e3}) and Lemma \ref{l6} have certain limitations which we had already mentioned in Remark \ref{r2}.

\end{section}
\hfill\\

\section*{Acknowledgments}
\noindent
First author is thankful to UGC for its supporting NET Junior Research Fellowship with UGC Ref. No. : 191620054184 \\
Second author is thankful to UGC for its supporting NET Junior Research Fellowship with UGC Ref. No. : 1004/(CSIR--UGC NET Dec. 2017). \\

\end{document}